\documentclass[11pt, reqno]{amsart}
\usepackage{amsmath,amsthm,amssymb,enumerate}
\usepackage[all]{xy}
\usepackage{array}
\usepackage{float}
\usepackage{graphicx}

\usepackage[dvipsnames]{xcolor}
\definecolor{citation}{rgb}{0,.40,.80}
\usepackage[colorlinks,linkcolor=citation,citecolor=citation,urlcolor=citation]{hyperref}
\usepackage{cleveref}

\newcommand \C {\mathbb C}

\newcommand \F {\mathbb F}
\newcommand \fF {\mathcal F}

\newcommand \I {\mathcal I}

\renewcommand \O {\mathcal O}
\renewcommand \P {\mathbb P}

\newcommand \Z {\mathbb Z}

\newcommand\xra{\xrightarrow}

\DeclareMathOperator \alb {alb}

\DeclareMathOperator \Hilb {Hilb}

\DeclareMathOperator \Pic {Pic}
\DeclareMathOperator \NS {NS}

\DeclareMathOperator{\Sp}{Sp}
\newcommand{\Spa}{\Sp(A[2])}

\DeclareMathOperator{\Sym}{Sym}

\newtheorem {thm} {Theorem}[section]

\newtheorem {lemma} [thm] {Lemma}
\newtheorem {prop} [thm] {Proposition}

\theoremstyle{definition}

\newtheorem {example}[thm] {Example}

\newtheorem {rmk}[thm] {Remark}

\numberwithin{equation}{section}

\makeatletter
\newcommand{\xleftrightarrow}[2][]{\ext@arrow 3359\leftrightarrowfill@{#1}{#2}}
\newcommand{\dashto}[2][]{\ext@arrow 0359\rightarrowfill@@{#1}{#2}}
\newcommand{\xdashleftarrow}[2][]{\ext@arrow 3095\leftarrowfill@@{#1}{#2}}
\newcommand{\xdashleftrightarrow}[2][]{\ext@arrow 3359\leftrightarrowfill@@{#1}{#2}}
\def\rightarrowfill@@{\arrowfill@@\relax\relbar\rightarrow}
\def\leftarrowfill@@{\arrowfill@@\leftarrow\relbar\relax}
\def\leftrightarrowfill@@{\arrowfill@@\leftarrow\relbar\rightarrow}
\def\arrowfill@@#1#2#3#4{%
  $\m@th\thickmuskip0mu\medmuskip\thickmuskip\thinmuskip\thickmuskip
   \relax#4#1
   \xleaders\hbox{$#4#2$}\hfill
   #3$%
}
\makeatother

\begin{document}

\author{Katrina Honigs}
\author{Graham McDonald}
\address{Department of Mathematics\\
Simon Fraser University\\
8888 University Drive\\
Burnaby, BC, V5A 1S6\\
Canada}

\title[Theta characteristics and the fixed locus of {$[-1]$}]{Theta characteristics and the fixed locus of~$[-1]$ on some varieties of Kummer type}

\date{}

\begin{abstract}
  We study some combinatorial aspects of the fixed loci of symplectic involutions acting on hyperk\"ahler varieties of Kummer type.

  Given an abelian surface $A$ with a $(1,d)$-polarization $L$, there is an
  isomorphism $K_{d-1}A\cong K_{\hat{A}}(0,\hat{l},-1)$ between a hyperk\"ahler of Kummer type that parametrizes length-$d$ subschemes of $A$ and one that parametrizes degree $d-1$ line bundles supported on curves in $|\hat{L}|$, where $\hat{L}$ is the dual $(1,d)$-polarization on $\hat{A}$.
We examine the bijection this isomorphism gives between isolated points in the fixed loci of $[-1_A]$ when $d$ is odd, which has a combinatorics related to theta characteristics.

Along the way, we give a table of numerical values for
a formula of \cite{KMO} counting the number of components
of a symplectic involution acting on a Kummer-type variety. 
\end{abstract}

\maketitle

\section{Introduction}\label{intro}

The geometry of involutions and their fixed loci is  important to the study of hyperk\"ahler varieties.
In the case of
fourfolds of Kummer type, it was shown by Hassett and Tschinkel \cite{HasTsc} that a portion of their cohomology is spanned by classes determined by such fixed loci.
In \cite{Frei_Honigs,FHV}, the authors use a connection between these classes and 
torsion points on abelian surfaces to explore the derived equivalence of
Kummer $4$-folds over number fields, but such equivalences
are still not completely understood.

In \cite{KMO},
Kamenova, Mongardi, and Oblomkov show that the
fixed locus of any symplectic involution on a Kummer-type variety
is the union of
finitely many 
hyperk\"ahler varieties of $\mathrm{K3}^{[n]}$-type, and 
give a formula enumerating its components.
These fixed loci have been studied in the case of generalized Kummer $4$-folds, but in general have not been explicitly described.

In this paper,
as a step toward understanding this geometry further, 
we examine how the isolated points in the fixed loci of symplectic involutions change across an isomorphism that arises from a Fourier--Mukai transform. In order to accomplish this task, we develop a combinatorial tool relating to the Weil pairing on the $2$-torsion points of an abelian surface with a polarization of odd degree, which is of independent interest.

\medskip

Given an abelian surface $A$,
the \textit{generalized Kummer variety} $K_n(A)$ is the $2n$-dimensional variety given by the
fiber of the summation map  $\Sigma:\Hilb^{n+1}(A)\to A$
over the identity point of $A$.
The variety $K_1A$ is the familiar Kummer K3 surface.
The involution on $K_n(A)$ induced by $[-1_A]:A\to A, a\mapsto -a$ is symplectic.
All of $K_1(A)$ is fixed by $[-1_A]$, but the fixed locus of the action of $[-1_A]$ on
the fourfold $K_2A$ consists of a
Kummer K3 surface as well as $36$ isolated points:
one is supported on the identity and the other $35$ are three distinct points in $A[2]$ \cite{HasTsc,Tari}.

If we fix a polarization $H$ on $A$, we may
also form Kummer-type varieties using $H$-stable sheaves on $A$.
For instance, 
if $A$ has a (symmetric) $(1,d)$-polarization $L$ 
and we take $l=[L]$, the points of the Kummer $2(d-1)$-fold
$K_A(0,l,s)$ parametrize certain
sheaves of rank $0$, N\'eron-Severi class~$l$ and Euler characteristic $s$, i.e.\
rank~ $1$ torsion free sheaves supported  on irreducible curves in the linear system $|L|$.
Generically, these sheaves are line bundles of degree $s+\frac{l^2}{2}=s+d$.
Pullback by $[-1_A]$ also gives a symplectic involution on $K_A(0,l,s)$. 

These two moduli constructions are closely related.
When $A$ has Picard rank $1$, 
there is an isomorphism
given by a Fourier--Mukai transform $\Psi$ where
$\Psi\circ[-1_A]=[-1_{\hat{A}}]\circ\Psi$:
\begin{equation}\label{psi}
\Psi:K_{d-1}(A)\xra{\sim}  K_{\hat{A}}(0,\hat{l},-1), 
\end{equation}  
where
$\hat{l}$ is the N\'eron-Severi class of a $(1,d)$-polarization $\hat{L}$ on the dual abelian variety $\hat{A}$ \cite{Gulbrandsen,Yoshioka}.
In the fourfolds case, all the symplectic involutions of these Kummer-type varieties are known to be translations of $[-1_A]$, so understanding
$[-1_A]$ 
captures the essential behavior of all symplectic involutions.

 The points in $K_{\hat{A}}(0,\hat{l},-1)$ fixed by $[-1_{\hat{A}}]$ are supported on curves from either of two eigenspaces of the linear system $|\hat{L}|$ under the action of $[-1_{\hat{A}}]$. 
When $d=3$, it was shown in  \cite{Frei_Honigs}
this distinction
partitions
the $36$ isolated points
into a set of $16$ and of $20$.
 However,
such a partition is not immediately obvious in the fixed locus of $K_2(A)$, even though
 $\Psi$ gives a bijection between these sets of isolated points.

 In this paper, we characterize the partition both numerically in terms of theta characteristics,
which are the quadratic forms on $A[2]$ associated to the Weil pairing determined by $L$,
 and as distinct orbits under the symplectic group of $A[2]$.
 We also use our methods to
identify, for $\xi\in K_{d-1}(A)$, singular points in the supporting curve of $\Psi(\xi)$.
The main insight in these results is the application of the work of Gulbrandsen \cite{Gulbrandsen} on the supporting curves of the images of $\Psi$ to this case.

\medskip

In the following statements of results, we take $A$ to be an abelian surface of Picard rank~$1$ with a symmetric $(1,d)$-polarization $L$ where $d$ is odd. We consider the isomorphism $\Psi$ as in \eqref{psi}.
Let $\xi=(u_1,\ldots, u_d)\in K_{d-1}A$ be a subscheme of $A$
that consists of $d$ distinct points in $A[2]$ and
is an isolated point in the fixed locus of $[-1_A]$ acting on
$K_{d-1}A$. Let $S'$ be the set of all such $\xi$.

\begin{thm}[Prop.~\ref{determinant}]\label{thm.one} Let $q$ be a theta characteristic associated with~$L$. For any $\xi\in S'$,
  the value of $q(\xi):=\sum_{i=1}^{d}q(u_i)\in \F_2$ determines which eigenspace of $[-1_{\hat{A}}]$ acting on
  $|\hat{L}|$ contains the supporting curve of $\Psi(\xi)$.
\end{thm}

When $d=3$, $S'$ consists of $35$ elements
$\xi=(u_1,u_2,u_3)\in K_2A$.
For $15$ of the $\xi\in S'$, $q(u_1)+q(u_2)+q(u_3)=0$ and for the other $20$, $q(u_1)+q(u_2)+q(u_3)=1$.

For any odd $d$, there is also a natural action of the symplectic group of $A[2]$ on $S'$, which respects the value of $q(\xi)$.
 In the case $d=3$, we have the following result.

\begin{thm}[Prop.~\ref{twoorbits}]
When $d=3$, the action of $\Sp(A[2])$ on $S'$
has two orbits, which coincide with the two possible values of $q(\xi)$.
\end{thm}

Our analysis of theta characteristics
can be visualized using the following table listing the points of $A[2]$, which we call a \textit{Hudson table}:
\[
\begin{tabular}{c|ccc} 
$1$ & $ab'$ & $bc'$ & $ca'$\\
\hline\vrule height 12pt width 0pt
$ac'$ & $a'$ & $c$ & $bb'$\\ 
 $ba'$ &$cc'$  & $b'$ & $a$\\ 
 $cb'$ & $b$ & $aa'$ &$c'$
\end{tabular}
\]
We name this table for its appearance in the classic text of Hudson \cite{Hudson} on the (singular) Kummer surfaces associated to  principally polarized abelian surfaces. There, the table is given in relation to the $(16,6)$ configuration of points and planes.
These configurations have a very rich combinatorics, which is explored in, for instance, Gonzalez-Dorrego \cite{Gonzalez}.
We observe that for any polarization of odd degree, the Hudson table captures information on the theta characteristics and we find it useful in our arguments. Its application in these contexts appears to be novel.

The characterization of the fixed locus of $[-1_{\hat{A}}]$ acting on $K_{\hat{A}}(0,\hat{l},-1)$ in \cite{Frei_Honigs} determines the set of all supporting curves of $\Psi(\xi)$ for $\xi\in S'$ when $d=3$, but does not indicate which individual elements $\xi\in S'$ matches with each curve. 
The results of Theorem~\ref{thm.one} allow us to now determine the supporting curve of $\Psi(\xi)$ for each $\xi\in S'$. Here is an example of a result that we phrase in terms of the Hudson table:

\begin{thm}[Proposition~\ref{Bx}]
 Let $d=3$ and $\xi=(u_1,u_2,u_3)\in S'$ such that $q(\xi)=1$.
There is a  unique $x\in A[2]$  in the Hudson table
that is distinct from the $u_i$ but shares a row or column with each of them.  The supporting curve of $\Psi(\xi)$ is the unique curve in $|\hat{L}|$ with one nodal singularity at the image of $x$ under the isomorphism $A[2]\simeq\hat{A}[2]$.
\end{thm}  

Our methods give information
about singular points of curves in the linear system $|\hat{L}|$
for higher (odd) values of $d$ as well, which we demonstrate by analyzing some  points of $K_4A$ in \S\ref{onefive}.

We also  numerical values and the generating function for the formula in \cite{KMO}
enumerating components of fixed loci on a variety of Kummer type.
The combinatorics of these fixed loci relate to other invariants of these varieties, such as cohomology (cf.~\cite{HasTsc}); these computations may be useful as a reference to researchers working in this~area.

\subsubsection*{Plan of paper}
In \S\ref{BN} we give further background on hyperk\"ahler varieties of Kummer type.
In \S\ref{formula} we give the values for the formula \cite[Theorem~3.9]{KMO}.
The remaining sections are focused on
results relating to the isomorphism \eqref{psi}:
In \S\ref{group} 
we discuss theta characteristics and how to visualize them in a Hudson table. 
In \S\ref{sec.iso}, we define the
isomorphism \eqref{psi} and prove results about the supporting curves of sheaves in the image of $\Psi$.
In \S\ref{fourfolds}, we review results on
the isolated points in the fixed locus of $[-1_A]$ and $[-1_{\hat{A}}]$ acting on each
of $K_2(A)$ and $K_{\hat{A}}(0,\smash{\hat{l}},-1)$ when $d=3$, and then
characterize the supporting curves of the images of the points in $K_2(A)$ under $\Psi$.
In \S\ref{sp} we examine the action of the symplectic group $\Spa$ on part of the fixed locus. Finally in \S\ref{onefive},
we give example applications 
of our methods to a $(1,5)$-polarized abelian surface.

\section*{Acknowledgements}
K.H.\ thanks Nils Bruin for helpful conversations at a formative stage of the project and was supported by an NSERC Discovery Grant.
G.M.\ was supported by an NSERC USRA.

\section{Background and notation}\label{BN}

Many of the varieties we consider are constructed using an abelian surface~$A$, which has dual abelian variety $\hat{A}$. We write the involution
$[-1_A]:A\to A, a\mapsto -a$.
As in \cite{Hudson},
we will use
juxtaposition to denote the group law on $A[2]$,
writing the identity point of $A$ as $1$, reserving $+$ for formal sums of divisors.
Given a line bundle $L$ on $A$, we denote its associated projectivized linear system by $|L|\cong \P H^0(A, L)$.
For simplicity, we work over the base field~$\C$, but it is possible to consider these constructions over other fields.

Hyperk\"ahler varieties are irreducible holomorphic symplectic varieties with a unique nondegenerate holomorphic two-form. An involution is referred to as ``symplectic'' if it respects this two-form. We encounter two deformation types of these varieties in this paper. Our primary focus is those of Kummer type, which are deformation equivalent to a generalized Kummer variety $K_nA$. We also encounter those of $\mathrm{K3}^{[n]}$ type, which are deformation equivalent to the Hilbert scheme of $n$ points on a K3 surface.

Some varieties of Kummer type can be constructed
as Albanese fibers of moduli of stable sheaves on an abelian surface. Let $A$ be an abelian surface with a polarization $H$ and a Mukai vector $v=(r,l,s)$.
The moduli space $M_A(v)$ parametrizes
$H$-stable sheaves on $A$ with Mukai vector $v$, that is, rank $r$, N\'eron-Severi class of the determinental line bundle $l$, and Euler characteristic $s$.
If $v$  is
primitive and positive where $v^2:=l^2-2rs\geq 0$ and $H$ is $v$-generic, then
the moduli space of stable sheaves $M_A(v)$ is a non-empty, smooth, projective and irreducible variety of dimension $v^2+2$.
We refer the interested reader to the papers \cite{Mukaisymp} and \cite{Yoshioka}, or the summary in \cite{Frei_Honigs} for more details on this construction and hypotheses.

In order to construct $K_A(v)$, it is necessary to use the Fourier--Mukai transform whose kernel is the Poincar\'e bundle $P_A$ of $A$, defined as follows:
\begin{align}
  \Phi_{P_A}:D(A)&\to D(\hat{A})\\
  \fF&\mapsto  Rp_{2*} (Lp_1^*(\fF)\otimes^{L} P_A)
\end{align}
The maps $p_1,p_2$ are the projections from $A\times \hat{A}$ to its factors. The functor $\Phi_{P_A}$ gives an equivalence between the bounded derived categories of coherent sheaves of $A$ and $\hat{A}$ \cite{Mukai}.

The Albanese variety of a moduli space $M_A(v)$ constructed as above is $A\times \hat{A}$. It is briefer to show the morphism to the Albanese torsor:
\begin{align}\label{alb}
\alb: M_A(v)&\to \Pic^l_A\times \Pic^{\hat{l}}_{\hat{A}}
\\ \notag
\fF&\mapsto (\det(\fF),\det(\Phi_P(\fF)))
\end{align}
where 
$\hat{l}$ is the N\'eron-Severi class of $\det(\Phi_P(\fF))$ for $\fF\in M_A(v)$,
which is  the negative of the Poincar\'e dual of $l$ \cite[Prop.~1.17]{MukaiFF}.
We define $K_A(v)$  to be a fiber of the Albanese morphism. All fibers are isomorphic, so without loss of generality, we may choose the convention that $K_A(v)$ is a fiber over a symmetric line bundle, in which case pullback by $[-1_A]$ is a symplectic involution on $K_A(v)$. 

In the case where $v=(1,0,-(n+1))$,
$K_A(v)$ is isomorphic to $K_{n}A$:
$K_A(1,0,-(n+1))$
parametrizes subschemes of length $n+1$ that sum to the identity in $A$.

The other moduli spaces of this type we will work with have Mukai vectors of the form $v=(0,l,s)$, where $\NS(A)=\Z l$ is the N\'eron-Severi class of a symmetric $(1,d)$-polarization $L$ on $A$.
For any $s\in \Z$,
the preimage of the Albanese morphism \eqref{alb}
on $M_A(0,l,s)$  over
$\{L\}\times \Pic^{\hat{l}}_{\hat{A}}$
is the relative compactified Jacobian
$\smash{\overline{\mathrm{Pic}}^{s+d}_{\mathcal{C}/|L|}}$ where $\mathcal{C}$ is the tautological family of curves in~$|L|$. 
Its elements 
consist of rank~$1$ torsion-free sheaves of Euler characteristic $s$
supported on the curves of the linear system $|L|$, which have arithmetic genus $d+1$.
When the supporting curves are 
smooth, these sheaves are line bundles of degree $s+d$.

Fix a symmetric line bundle $M\in \Pic^{\hat{l}}(\hat{A})$. 
We may then take $K_A(0,l,s)\subset\smash{\overline{\mathrm{Pic}}^{s+d}_{\mathcal{C}/\mathbb{P}|L|}}$ to be the fiber of the Albanese map \eqref{alb} over $(L,M)$.
Applying the universal property of the Jacobian
to the inclusion $C\hookrightarrow A$,
when $C$ is smooth,
induces a map  $j_C:\Pic^0_C\to A$. This map sends a line bundle $\O(D)$ on $C$ to the sum $\Sigma D$, which we obtain by summing points on the curve using the group law of $A$.
It was shown in \cite[\S 6]{Frei_Honigs} that for any smooth curve $C\in |L|$,
the points in $K_A(0,l,s)$ supported on $C$ are a translation of the fiber of $j_C$ over the identity in $A$. When $C$ is a singular curve in $|L|$,
this notion extends in a natural way
using the theory of generalized divisors (cf.~\cite{Kass}).

\section{A formula for the number of fixed loci}\label{formula}

In \cite[Theorem~3.9]{KMO}, the authors give a formula for the number of components of each dimension in the
fixed locus 
of a symplectic involution on a Kummer-type variety.
To find this formula, they consider the fixed locus of $[-1_A]$ acting on $K_nA$ and relate it to
points on $\Sym^{n+1}A$ that are supported on 
combinations of pairs of points $(p,-p)$ for $p\in A$ and tuples of points in $A[2]$ that sum to the identity.
The possible configurations of these supporting points dictate the
number of components of each dimension.
Thus, their  formula depends on the 
combinatorics problem of
choosing $n$ elements of $\F_2^4$ that sum to $0$.

In the following formula, the sum is taken over all subsets $I\subseteq \F_2^4$ that sum to $0$, $|I|$ denotes the number of elements of $I$, and
$\binom{n}{k}=0$ if $k>n$, $k<0$, or $k$ is not an integer. 

\begin{thm}[{{\cite[Theorem 3.9]{KMO} Kamenova, Mongardi, Oblomkov}}]
  Let $X$ be a $2n$-dimensional hyperkahler manifold of Kummer type, and let $\iota$ be a symplectic involution on $X$ that acts nontrivially on $H^3(X)$.
  Define
\begin{equation}\label{original}
  N_m^n:=\sum_{I,\sum I=0}\binom{|I|}{\frac{n-|I|}{2}-m}
\end{equation}  
The fixed locus of $\iota$ consists of $N^{n+1}_m$ connected components of dimension~$2m$. 
\end{thm}

The reference has a typographical error
switching
the top and bottom of the binomial notation in \eqref{original}.

Let $\gamma_i$ be number of ways to choose $0\leq i\leq 16$ elements of $\F_2^4$ that sum to $0$. Its values are listed in Table~\ref{gammai}.
We may write $N_n^m$ in terms of the $\gamma_i$ as follows:
\begin{equation}
	N_m^n=\sum_{i=0}^{16}\gamma_i\cdot\binom{i}{\frac{n-i}{2}-m}
	\label{eq: final_count_formula}
      \end{equation}

The values of $\gamma_i$ are given by the following generating function in the case $r=4$.

\begin{thm}[Song \cite{OEIS}]
  The following is the generating function for the number of subsets $I$ of the vector space $\mathbb{F}_2^r$
 so that $|I|=i$ and $\sum_{j\in I}j=0$: 
	\begin{equation}
		(x+1)^{2^{r-1}}\left(\sum\limits_{i=0}^{2^{r-1}}\binom{2^{r-1}}{2i}x^{2i}-(2^{r-1}-1)x\prod\limits_{k=0}^{r-2}\sum\limits_{i=0}^{2^k}\binom{2^k}{2i}x^{2i}\right)
		\label{eq: Jian_song}
	\end{equation}
\end{thm}

In the case $r=4$ we have:
\begin{align}\label{eq: Jian_song_n_4}
&(x+1)^8\left(\sum\limits_{i=0}^8\binom{8}{2i}x^{2i}-7x\prod\limits_{k=0}^2\sum\limits_{i=0}^{2^k}\binom{2^k}{2i}x^{2i}\right)
\\\notag
&=x^{16}+x^{15}+35x^{13}+140x^{12}+273x^{11}+448x^{10}+715x^9\\
&\phantom{{}=x^{16}}+870x^8+715x^7+448x^6+273x^5+140x^4+35x^3+x+1
  \label{eq: poly_expanded}
\end{align}

\begin{table}[H]
  \centering
\caption{Values of $\gamma_i$ }
	\begin{tabular}{|c|c|c|c|c|c|c|c|c|c|c|}
		\hline
		$i$&0&1&2&3&4&5&6&7\\
		\hline 
		$\gamma_i$&1&1&0&35&140&273&448&715\\
          \hline
                \hline
          8&9&10&11&12&13&14&15&16\\
          \hline
          870&715&448&273&140&35&0&1&1\\
          \hline
	\end{tabular}
	\label{gammai}
      \end{table}

\begin{rmk}
There is a symmetry $\gamma_i=\gamma_{16-i}$ because the sum of all elements in $\F_2^4$ is $0$: a choice of $i$ elements that sum to $0$ partitions $\F_2^4$ into two sets that sum to $0$.
\end{rmk}

Now we use the formula \eqref{eq: final_count_formula} and the values in Table~\ref{gammai} to give
values of $N_m^{n+1}$, i.e., 
numbers of components of each dimension of the fixed locus
of a symplectic involution acting on a $2n$-dimensional variety of Kummer type. We show these values up to $n=10$. 

We give the values in two tables, one for each parity of $n$, since, if $n$ is even (odd), only the $\gamma_i$ with $i$ odd (even)
contribute to $N_m^{n+1}$.
The components of the fixed locus of a symplectic involution 
on a $2n$-dimesional variety of Kummer type
have dimensions $2m$
for $\max\{0, \frac{n+1}{2} - 24\} \leq m \leq \frac{n+1}{2}$ if $n$ is odd, and $\max\{0, \frac{n}{2} - 22\} \leq m \leq \frac{n}{2}$ if $n$ is even since all the binomial coefficients in the formula \eqref{eq: final_count_formula}
will be $0$ outside these bounds. We leave the entries of the tables blank for values of $m$ not within these bounds.

\begin{table}[H]
	\centering
	\caption{Values of $N_m^{n+1}$, for $n$ odd}
	\begin{tabular}{|c|c|ccccccc|}
	\hline 
	&$n$ &&&&&&&\\
		\hline 
$m$:&&$0$&$1$&$2$&$3$&$4$&$5$&$\cdots$\\
		\hline 
        &$1$ 
             &$0$&$1$&&&&&\\
        & $3$ 
             &$140$&$0$&$1$&&&&\\
        &$5$
             &$1008$&$140$&$0$&$1$&&&\\
        &$7$
             &$4398$&$1008$&$140$&$0$&$1$&&\\
        &$9$
             &$14688$&$4398$&$1008$&$140$&$0$&$1$&\\
		&\vdots&\vdots&\vdots&\vdots&\vdots&\vdots&\vdots&$\ddots$\\
		\hline
	\end{tabular}
	\label{tab: even_n}
\end{table}

\begin{table}[H]
	\centering
\caption{Values of $N_m^{n+1}$, for $n$ even}
	\begin{tabular}{|c|c|ccccccc|}
		\hline
		&$n$&&&&&&&\\
		\hline
$m$:&&$0$&$1$&$2$&$3$&$4$&$5$&$\cdots$\\
		\hline 
		&$2$ 
                    &$36$&$1$&&&&&\\
		&$4$
                    &$378$&$36$&$1$&&&&\\
		&$6$
                    &$2185$&$378$&$36$&$1$&&&\\
		&$8$
                    &$8485$&$2185$&$378$&$36$&$1$&&\\
		&$10$
                    &$24453$&$8485$&$2185$&$378$&$36$&$1$&\\
		&\vdots&\vdots&\vdots&\vdots&\vdots&\vdots&\vdots&$\ddots$\\
		\hline
	\end{tabular}
	\label{tab: odd_n}
      \end{table}

For example, $N_0^{n+1}$ is the number of isolated fixed points on a $2n$-fold; $n=47$ is the largest value for which $N_0^{n+1}$ is nonzero. The formula
\eqref{eq: final_count_formula} in this case gives the following:
\[
  N_0^{n+1}=\gamma_{n+1}\binom{n+1}{0}+\gamma_{n-1}\binom{n-1}{1}+
\gamma_{n-3}\binom{n-3}{2}+\cdots,
\]
where $\gamma_{n+1}$ is the number of ways to choose $n+1$ distinct points in $\F_2^4$ that sum to $0$.
For any $n$, the component of the fixed locus with the largest dimension is deformation equivalent to one copy of $A^{[\lceil\frac{n}{2}\rceil]}$.
In the case of $[-1_A]$ acting on $K_nA$, $\gamma_{n+1}$ is counting the number of fixed points supported on $n+1$ distinct points of $A[2]$.
When $n=1$, the entire K3 surface
$K_1A$ is itself fixed. 
When $n=2$, the fixed locus of
$K_2A$ consists of $1$ fixed K3 surface as well as 
$36$ isolated fixed points, where 
$\gamma_3=35$ of those fixed points are supported on three distinct points in $A[2]$.

\section{Theta characteristics on
  abelian surfaces with odd degree polarizations}
\label{group}

In this section, we introduce theta characteristics and Hudson tables.

We begin by giving a brief review of the classical $(16,6)$ configuration of planes and points on the singular Kummer K3 surface associated to a principally polarized abelian surface. In particular, we examine an incidence table of the points and planes. We then show in \S\ref{label} that this table
captures information about theta characteristics for any abelian surface with
a polarization
of odd degree, and use this combinatorial tool to show identities among the quadratic forms.

In \S\ref{base} we review the relationship between theta characteristics
and the base locus of
the linear system of the polarization on $A$.

\subsection{The principally polarized case}

In \cite[Ch.~1 \S3]{Hudson},
Hudson describes a $(16,6)$ configuration in the singular Kummer K3 surface $A/[-1]$
associated to a principally polarized abelian surface $A$:
there are $16$ singular points, coming from $A[2]$,
which are in bijection with $16$ planes.
Each plane contains exactly six of the singular points and each point lies on six planes. One such plane is shown in Figure~\ref{fig: singular kummer}.

\begin{figure}[h]\label{fig: singular kummer}
\includegraphics[width=0.45\textwidth]{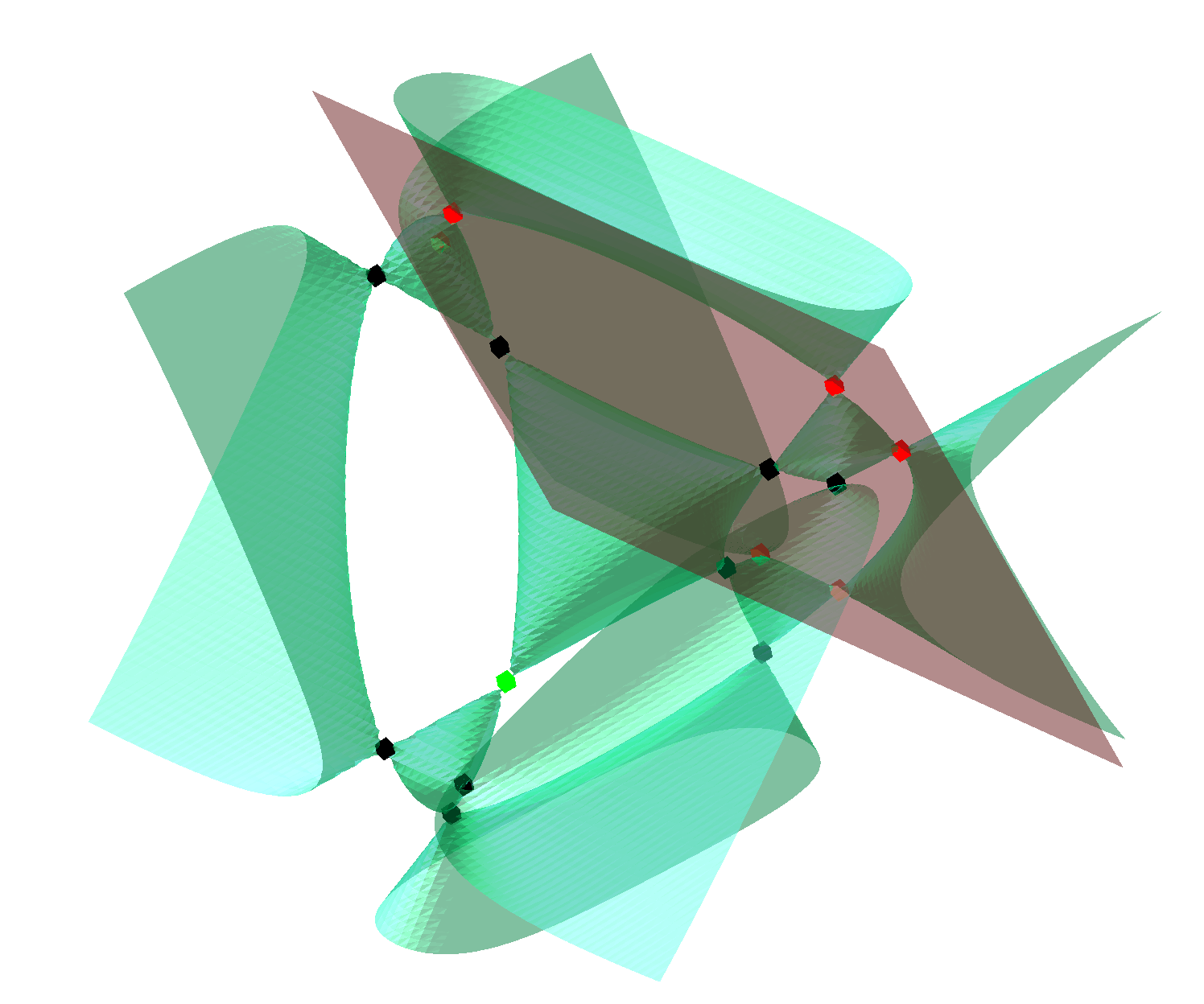}\quad\quad
\includegraphics[width=0.45\textwidth]{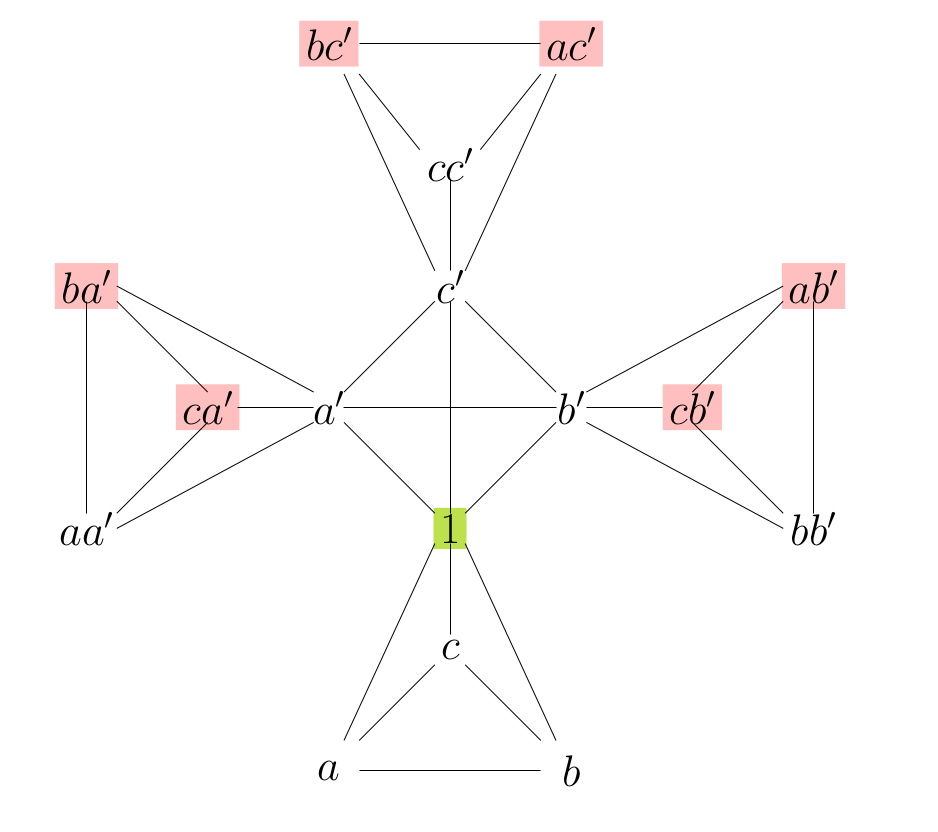}  
\caption{On the left is a
  singular Kummer surface in green  with a plane in red.
The $16$ singular points are marked. The point in green is dual to the plane. The six points intersecting the plane are red.
On the right we show a simplified diagram of this surface
with the singular points labelled so the reader can see them more clearly.}
\end{figure}

Hudson names a (non-minimal) set of generators for the group $A[2]$: $1,a,b,c,a',b',c'$. They obey the following multiplication tables:
\[
    \begin{tabular}{>{$}l<{$}|*{6}{>{$}l<{$}}}
    ~ &  a  & b & c  \\
    \hline\vrule height 12pt width 0pt
    a &  1 & c   & b \\
    b &     & 1   & a \\
    c &    &   & 1 \\
    \end{tabular}
\quad\quad\quad
    \begin{tabular}{>{$}l<{$}|*{6}{>{$}l<{$}}}
    ~ &  a'  & b' & c'  \\
    \hline\vrule height 12pt width 0pt
    a' &  1 & c'   & b' \\
    b' &     & 1   & a' \\
    c' &    &   & 1 \\
    \end{tabular}
\]
The table in Figure~\ref{mult}
contains all of the points of $A[2]$
and is also a multiplication table.
We call  it the \emph{Hudson table} of the $(16,6)$ configuration.
The Hudson table records the information of the  $(16,6)$ configuration: 
If we fix a plane and consider the corresponding point in the Hudson table, the six points lying on that plane are precisely those lying in either the same row or the same column. Similarly, if we fix a point of $A[2]$ and consider its entry in the Hudson table, the values lying
in either the same row or the same column correspond to the six planes on which that point lies.

\begin{figure}[H]
\begin{tabular}{c|ccc} 
$1$ & $ab'$ & $bc'$ & $ca'$\\
\hline\vrule height 12pt width 0pt
$ac'$ & $a'$ & $c$ & $bb'$\\ 
 $ba'$ &$cc'$  & $b'$ & $a$\\ 
 $cb'$ & $b$ & $aa'$ &$c'$
\end{tabular}
\caption{The Hudson table}
\label{mult}
\end{figure}

For example, the plane corresponding to the point $1$ contains the points $ac',ba',cb',ab',bc',ca'$, and the point $1$ lies on the planes corresponding to $ac',ba',cb',ab',bc',ca'$.

\subsection{$(1,d)$-polarized surfaces with $d$ odd}\label{label}
Let $L$ be a (symmetric) $(1,d)$-polarization of $A$  
with $d$ odd.
We can generalize the incidence structure of points and planes described in the previous section using theta characteristics on $A[2]$. See Gross and Harris \cite[\S1]{GrossHarris} and Farkas \cite[\S1,2]{Farkas} for modern introductions to the subject. 

Since $d$ is odd, the polarization
$\phi_L:A\to \hat{A}$ restricts to 
an isomorphism $A[2]\cong \hat{A}[2]$ and hence the Weil pairing gives a non-degenerate strictly alternating form $\langle\, ,\,\rangle$ on $A[2]$ as an $\F_2$-vector space.
The quadratic forms associated to this pairing on $A[2]$ are the
\textit{theta characteristics} of our polarized abelian surface (cf.~\cite[\S1]{BolognesiMassarenti}), which form
a principal homogeneous space over $A[2]$.

From the general theory of such pairings, it is possible to decompose
$A[2]$ into a sum of two maximal isotropic vector spaces $A[2]\cong X\oplus Y$.
Given such a decomposition, any elements both in $X$ or both in $Y$ pair to $0$, and
the function $q:A[2]\to \F_2$ defined by
$q(x\cdot y)=\langle x,y\rangle$ for  $x\in X, y\in Y$ is an even quadratic form. It is possible to choose dual symplectic bases $\{x_i\}_i,\{y_i\}_i$ for $X$ and~$Y$, so that $\langle x_i,y_j\rangle=\delta_{ij}$.

If we label $X=\{a,b\}$ and $Y=\{b',a'\}$ and call the associated quadratic form $q_1$,
we find that the theta characteristics have a pleasant visual description in terms of the Hudson table of \Cref{mult}, which we now explain.
We partition the points of $A[2]$ into two sets depending on the values $q_1$ takes on them:
\begin{align}\label{even}
  10_1&:=\{v\in A[2]\mid q_1(v)=0\} =\{1,a,b,c,a',b',c',aa',bb',cc'\} \\
  6_1&:=\{v\in A[2]\mid q_1(v)=1\}=\{ab',ac',ba',bc',ca',cb'\}
  \label{odd}
\end{align}
We have chosen notation for these subsets of $A[2]$ using 
the number of elements in the set with a subscript matching $q_1$.
The elements of $6_1$ are the entries of the Hudson table
that are in the same row or column as $1$.

Now that we have chosen $q_1$ to correspond to $1\in A[2]$, the structure of the theta characteristics as a principal homogeneous space of $A[2]$ determines a theta characteristic $q_v$ for each $v\in A[2]$:
\begin{align}\label{qv}
  q_v:A[2]&\to \F_2\\\notag
  u&\mapsto q_1(u)+\langle v,u\rangle
\end{align}
If $v\in 10_1$, $q_v$ has Arf invariant $0$ and thus is even,
and if $v\in 6_1$, $q_v$ has Arf invariant~$1$ and is odd.
For each $v\in A[2]$, $q_v$ partitions the points of $A[2]$ into the following sets, depending on whether $q_v$ is even or odd:
\begin{align}
\text{$q_v$ even: }  10_v&:=\{v\in A[2]\mid q_1(v)=0 \}, \quad 6_v:=\{v\in A[2]\mid q_1(v)=1\}\\
\text{$q_v$ odd: }  10_v&:=\{v\in A[2]\mid q_1(v)=1 \}, \quad  6_v:=\{v\in A[2]\mid q_1(v)=0\}
\end{align}
Regardless of the choice of $v$, the elements in $6_v$ are precisely the points of $A[2]$, other than $v$ itself,
that are in the same row or column of the Hudson table as $v$. 
In the following table, the points in $6_{b'}$ are highlighted. The non-highlighted points are $10_{b'}$. 
\begin{equation}\label{tab: hudson's incidence diagram highlighted}
	\begin{tabular}{c|ccc}
		$1$&$ab'$&\colorbox{pink}{$bc'$}&$ca'$\\
		\hline
		$ac'$&$a'$&\colorbox{pink}{$c$}&$bb'$\\
		\colorbox{pink}{$ba'$}&\colorbox{pink}{$cc'$}&$b'$&\colorbox{pink}{$a$}\\
		$cb'$&$b$&\colorbox{pink}{$aa'$}&$c'$
	\end{tabular}
\end{equation}

The following lemmas will be useful in later sections.

\begin{lemma}\label{sixten}
For any $v\in A[2]$, $v6_v=6_1$ and $v10_v=10_1$.
\end{lemma}

\begin{proof}
Since the sets $10_v$ and $6_v$ are
  characterized by the fact that $q_v$ is constant on each of them,
  it suffices to show that 
  for any $u_1,u_2\in A[2]$, if $q_v(u_1)=q_v(u_2)$, then $q_1(vu_1)=q_1(vu_1)$.

Using \eqref{qv} and the relationship between quadratic forms and the bilinear pairing, we have for each $i\in\{1,2\}$:
\[q_v(u_i)=q_1(u_i)+\langle v,u_i\rangle,\quad
q_1(vu_i)+q_1(v)+q_1(u_i)=\langle v,u_i\rangle.
\]
Combining these equations, we have 
$$q_v(u_i)=q_1(u_i)+q_1(vu_i)+q_1(v)+q_1(u_i)=q_1(vu_i)+q_1(v),
$$
proving the result.
\end{proof}

\begin{lemma}\label{prod}
For any distinct $u,v\in A[2]$, $u$ and $v$ share the same row or column in the Hudson table if and only if $q_1(uv)=1$ (equivalently $uv\in 6_1$). 
\end{lemma}

\begin{proof}
  The points $u$ and $v$ are in the same row or column of the Hudson table if and only if $u\in 6_v$. By Lemma~\ref{sixten}, $u\in 6_v$ if and only if
  $uv\in 6_1$. 
\end{proof}  

\begin{lemma}\label{qform}
  Let $u_1,\ldots,u_n\in A[2]$ so that $u_1\cdots u_n=1$.
  Then the sum $q_v(u_1)+\cdots+q_v(u_n)$  has the same value
for any choice of $v\in A[2]$.
\end{lemma}  

\begin{proof} Let $v\in A[2]$.
Applying \eqref{qv}, we have
$$q_1(u_1)+\cdots+q_1(u_n)=q_v(u_1)+\cdots+q_v(u_n)+
\langle v,u_1 \rangle+\cdots+\langle v,u_n \rangle.
$$
Using bilinearity, 
$\langle v,u_1 \rangle+\cdots+\langle v,u_n \rangle
=\langle v,u_1\cdots u_n \rangle=\langle v,1 \rangle=0$, and thus
$q_1(u_1)+\cdots+q_1(u_n)=q_v(u_1)+\cdots+q_v(u_n)$.
\end{proof}

\subsection{$2$-torsion base points of a polarization}\label{base}
Let $A$ be an abelian surface with a symmetric polarization $L$
of odd degree.
The $16$ possible choices of the line bundle $L\otimes P_x$ for $x\in \hat{A}[2]$ correspond to the theta characteristics described in the previous section. 

Following \cite[\S2]{BolognesiMassarenti}, we describe
the eigenspaces of $[-1_A]$ acting on $H^0(A,L\otimes P_x)$ and
the points of $A[2]$ in the base loci of these linear systems. The eigenvalues of the action are $1,-1$.

Fix a line bundle $L\otimes P_x$. It corresponds to a theta characteristic $q_v$.
If $q_v$ is even, then the $+1$-eigenspace
$H^0(A,L\otimes P_x)^+$ is
$\frac{d+1}{2}$-dimensional with
 base locus 
$6_v\subset A[2]$, and
the $-1$-eigenspace $H^0(A,L\otimes P_x)^{-}$ is
$\frac{d-1}{2}$-dimensional with
base locus
$10_v\subset A[2]$. If $q_v$ is odd,
$H^0(A,L\otimes P_x)^-$ is
$\frac{d+1}{2}$-dimensional with base locus $6_v\subset A[2]$ 
and 
$H^0(A,L\otimes P_x)^{+}$ is
$\frac{d-1}{2}$-dimensional with base locus $10_v\subset A[2]$.

\section{Isomorphism of moduli spaces}\label{sec.iso}

In this section, we study the isomorphism
$\Psi:K_{d-1}A\xra{\sim} K_{\hat{A}}(0,\hat{l},-1)$.
In \S\ref{iso} and \S\ref{curve}, we briefly define $\Psi$ and then examine the supporting curves of $\Psi(\xi)$ for $\xi\in K_{d-1}A$.
In \S\ref{notiso}, we give criteria that allows us to find points of $\hat{A}[2]$ in these supporting curves.

\subsection{The isomorphism}\label{iso}
Let $A$ be an abelian surface that is $(1,d)$-polarized by a symmetric line bundle $L$ so that $\NS(A)=\Z l$, where $l=[L]$. Let $\hat{l}\in \NS(\hat{A})$ as in \S\ref{BN}.

By work of Gulbrandsen
\cite{Gulbrandsen} and Yoshioka \cite[Prop.~3.5]{Yoshioka}, there is an isomorphism
$\Psi:K_{d-1}A\xra{\sim} K_{\hat{A}}(0,\hat{l},-1)$, which is given by the following composition:
\begin{align}
  \Psi:  K_{d-1}A &\xra{\sim} K_A(1,0,-d) \xra{\sim} K_A(1,l,0)  \xra{\sim}
                K_{\hat{A}}(0,\hat{l},-1)
\\\notag                                                          
  \xi&\rlap{${}\mapsto \mathcal{I}_{\xi}$}
       \phantom{\xra{\sim} K_A(1,0,-d)}\!
       \llap{$I$}\rlap{${}\mapsto I\otimes L$}
       \phantom{\xra{\sim} K_A(1,l,0)}\,\,\,\,
\llap{$\fF$}{}\mapsto \Phi_P(\fF)[-1] 
\end{align}
The leftmost isomorphism sends a subscheme of length $d$ to its ideal sheaf.
Tensoring with the line bundle $L$ gives the middle isomorphism.
The morphism on the right is given by the action of the Fourier--Mukai transform $\Phi_{P_A}$ on the sheaves parametrized by $K_A(1,l,0)$. The images
under $\Phi_{P_A}$
are, \textit{a priori}, complexes of sheaves, but they are 
supported only in index $1$, so shifting the indexing by $-1$ gives a well-defined map. 

Since $L$ is symmetric, $\Psi\circ[-1_A]=[-1_{\hat{A}}]\circ\Psi$, and therefore $\Psi$ restricts to an isomorphism between the fixed loci of $[-1_A]$ and $[-1_{\hat{A}}]$. We study the bijection $\Psi$ gives on the isolated points in these fixed loci, which we call $S$ and $R$, respectively.
\begin{equation}\label{ST}
\xymatrix@R=1em@C=1em{  
  \llap{${\Psi:{}}$}S \ar[r]^{\sim} \ar@{}[d]|{\textstyle\cap}  &R \ar@{}[d]|{\textstyle\cap}\\
K_{d-1}A &K_{\hat{A}}(0,\hat{l},-1)}
\end{equation}
We will often restrict our focus to the subset $S'\subset S$
whose elements are subschemes of $A$ consisting of $d$ distinct points. 
In the notation of \S\ref{formula}, $\# S=\# R=N_0^d$ and 
$\# S'=\gamma_d$.

\subsection{The support of $\Psi(\xi)$}\label{curve}

The points of $K_{\hat{A}}(0,\hat{l},-1)$  parametrize certain rank $1$ torsion-free sheaves (see \S\ref{BN}) supported on curves in the linear system $|\smash{\hat{L}}|$, where $\smash{\hat{L}}\in \smash{\Pic^{\hat{l}}(\hat{A})}$ is a $(1,d)$-polarization on $\smash{\hat{A}}$.
In this section, we follow the approach of \cite[3.1]{Gulbrandsen} (cf.~\cite[11.3]{Polishchukbook})
to describe the support of $\Psi(\xi)$ for $\xi\in K_{d-1}A$.

For any 
$\xi\in K_{d-1}A$
and $x\in \hat{A}$, the fiber of $\Psi(\xi)$ over $x$ is the following:
\[
\Psi(\xi)\otimes k(x)\cong H^1(A,\I_\xi\otimes L\otimes P_x),
\]
where $P_x$ is the line bundle on $A$ corresponding to $x$.
The set of points in the curve supporting  $\Psi(\xi)$, 
which we call $D_{\Psi(\xi)}$,
is therefore the following: 
\begin{equation}
\{x\in\hat{A}\mid H^1(A,\I_{\xi}\otimes L\otimes P_x)\neq 0\}.
\end{equation}
To better understand
the vanishing of 
$H^1(A,\I_\xi\otimes L\otimes P_x)$, we consider the short exact sequence
\[
  0\to \I_{\xi}\to \O_A\to
\O_{\xi}\to 0,
\]
where $\O_{\xi}$ is the structure sheaf associated to $\xi$.
When $\xi$ consists of $d$ distinct points, $\O_{\xi}$ is the direct sum of $d$ skyscraper sheaves.

After tensoring with $L\otimes P_x$ we have
\[
  0\to \I_{\xi}\otimes L\otimes P_x\to L\otimes P_x\to
\O_\xi\otimes L\otimes P_x \to 0.
\]
By Mumford's index theorem, 
$H^1(A,L\otimes P_x)=0$. Therefore,
applying the global sections functor gives the following
long exact sequence:
\begin{align}\notag
  0\to H^0(A,\I_{\xi}\otimes L\otimes P_x)
  \to H^0(A,L\otimes P_x)
  \xra{\hypertarget{star}{(\star)}}
  H^0(A,\O_{\xi}\otimes L\otimes P_x)
\\\label{les}
  {}\to
  H^1(A,\I_{\xi}\otimes L\otimes P_x)\to 0.
\end{align}
This long exact sequence gives us a new way to characterize $D_{\Psi(\xi)}$.
The cohomology groups $H^1(\I_{\xi}\otimes L\otimes P_x)$ and $H^0(\I_{\xi}\otimes L\otimes P_x)$ have the same dimension, and
$H^0(\I_{\xi}\otimes L\otimes P_x)$ is nonzero exactly when $\xi$ lies on a curve in the linear system $|L\otimes P_x|$. Equivalently, $x\in D_{\Psi(\xi)}$ if and only if \hyperlink{star}{$(\star)$} is not an isomorphism.

\subsection{Criteria for $(\star)$ to not be an isomorphism}\label{notiso}

In this section,
we assume that
$d$ is odd, $x\in \hat{A}[2]$ and $\xi\in S'\subset K_{d-1}A$.
We will give criteria under which \hyperlink{star}{$(\star)$} is not an isomorphism 
by comparing the action of $[-1_A]^*$ on its domain and codomain.

For any $\xi\in S'$ and $q_v$ a theta characteristic of our polarized abelian surface, we may define the following value:
\begin{equation}
q(\xi):=\textstyle\sum_{u\in \xi} q_v(u)\in \F_2.
\end{equation}
By Lemma~\ref{qform}, $q(\xi)$ is independent of the choice of $q_v$. 

In the table below, we give some numbers of
points $\xi\in S'$ having each value of $q(\xi)$: 
\begin{equation}\label{zero_one}
\begin{tabular}{c|c|c|c}
  $d$&$\gamma_d$&$\#q(\xi)=0$&
                 $\#q(\xi)=1$\\\hline
	$3$&35&15&20\\
	$5$&273&141&132\\
  $7$&715&355&360\\
  $9$&715&355&360\\
  $11$&273&141&132\\
  $13$&35&15&20\\
\end{tabular}  
\end{equation}

In light of \S\ref{base}, the following result allows us to determine which elements $\xi\in S'$ have supporting curves $D_{\Psi(\xi)}$ in each eigenspace of $|\hat{L}|$.

\begin{prop}\label{determinant}
Let $\xi\in S'\subset K_{d-1}A$. 

Let $d\equiv 1\mod 4$.
  If $q(\xi)=0$ then $6_1\subset D_{\Psi(\xi)}$ and
  if $q(\xi)=1$ then $10_1\subset D_{\Psi(\xi)}$.

  Let   $d\equiv 3\mod 4$.
  If $q(\xi)=0$ then $10_1\subset D_{\Psi(\xi)}$ and
  if $q(\xi)=1$ then $6_1\subset D_{\Psi(\xi)}$.  
\end{prop}

\begin{proof}
We first treat the case where $d\equiv 3\mod 4$. Then $\frac{d+1}{2}$ and $\frac{d-1}{2}$ are even and odd, respectively.

From the discussion in \S\ref{base}, we have that if 
$L\otimes P_x$ is even, the determinant of the action of $[-1_A]^*$ on  $H^0(L\otimes P_x)$ is $-1$. If $L\otimes P_x$ is odd, the determinant of the action of $[-1_A]^*$ on  $H^0(L\otimes P_x)$ is $1$.

The line bundle $L\otimes P_x$ corresponds to a quadratic form $q_v$. For any $z\in A[2]$, the action of $[-1_A]^*$ on the skyscraper $k(z)\otimes L\otimes P_x$ is $(-1)^{q_v(z)}$. The multiplicities of the eigenvalues $1$ and $-1$ of the action of $[-1_A]^*$ on $H^0(A,\O_\xi\otimes L\otimes P_x)$ are
$\#\{i\mid q_v(z_i)=0\}$ and $\#\{i\mid q_v(z_i)=1\}$, respectively.
Thus the determinant of the action of $[-1_A]^*$ on $H^0(A,\O_\xi\otimes L\otimes P_x)$ is given by $(-1)^{q(\xi)}$, which happens to be independent of the choice of $x\in \hat{A}[2]$.

If the determinants of $[-1_A]^*$ acting on $H^0(A,L\otimes P_x)$ and on $H^0(\O_\xi\otimes L\otimes P_x)$ differ, then \hyperlink{star}{$(\star)$} is not an isomorphism. These determinants must differ if
 $q(\xi)=0$ and $L\otimes P_x$ is even or if
 $q(\xi)=1$ and $L\otimes P_x$ is odd, hence our result.

 When  $d\equiv 1\mod 4$, $\frac{d+1}{2}$ and $\frac{d-1}{2}$ are odd and even. The same line of reasoning as above then gives the result.
\end{proof}

\begin{rmk}
  In the case where $d=3$, if $\xi=(u,v,w)\in S'$, then
  $q(\xi)$ is $0$ or $1$ if and only if $q_u,q_v,q_w$ is a syzygetic or azygetic triple, respectively (see \cite[Def.~1.4]{Farkas}).
\end{rmk}  

\begin{rmk}
The proof of Proposition~\ref{determinant} uses the assumption that $\xi$ consists of $d$ distinct points in computing the action
  of $[-1_A]$ on $H^0(\O_{\xi}\otimes L\otimes P_x)$.
When $\xi$ has fewer supporting points,
 this statement cannot be generalized by
simply adding a multiplicity in the formula
 $q(\xi)$ according to length:
If $\xi$ consists of the point $1$ and a length $2$ subscheme of $A$ supported at some $a\in A[2]$, then 
such a formula would give $q(\xi)=0$. However, points $\xi$ of this description are contained in a Kummer K3 surface 
(\cite{HasTsc}, cf.~\S\ref{formula} p.~8) whose image under $\Psi$ is fibered over the eigenspace of $|L|$ with base locus $6_1$
 (cf.~\cite{Frei_Honigs}). 
\end{rmk}

\begin{rmk}\label{singular}
  If the kernel of \hyperlink{star}{$(\star)$} is $1$-dimensional, then $x$ is a smooth point on the curve $D_{\Psi(\xi)}$, but if the kernel has a higher dimension, $x$ cannot be a smooth point.
Under the criteria in Proposition~\ref{determinant}, the dimension of the kernel of \hyperlink{star}{$(\star)$} is odd,
 and so the supporting points may be smooth. 
However, in cases where \hyperlink{star}{$(\star)$} is not an isomorphism but the domain and codomain have equal determinants, the supporting point $x$ must be singular in $D_{\Psi(\xi)}$.
\end{rmk}
  
\section{The fourfolds case}\label{fourfolds}

In this section, we fix the following data: Let $A$ be an abelian surface that is $(1,3)$-polarized by a symmetric line bundle $L$ so that $\NS(A)=\Z l$, where $l=[L]$. We assume that $L$ is associated to an even theta characteristic $q_1$.

We examine the bijection \eqref{ST} in this case.  We first review the elements of the sets $S$ and $R$ and then apply the results of Section~\ref{sec.iso}, which allow us to deduce which curve in $|L|$ supports $\Psi(\xi)$ for each $\xi\in S$.

\subsection{The set $S$}\label{isok2} 

As summarized in \S\ref{intro}, the elements of $S$
consist of one
subscheme of $A$ whose support is the identity and $35$ subschemes consisting of
three distinct points in $A[2]$ that sum to the identity. Using the notation for $A[2]$ established in \S\ref{group},
we list the $35$ points in the first and third columns of Table~\ref{correspond} according to the  value of $q(\xi)$.

\subsection{The set $R$}\label{isokl}

We defined $R$ to be a subset of $K_{\hat{A}}(0,\hat{l},-1)$.
However, the choice of
$L$ allows us to identify the points of $A[2]$ on curves in $|L|$ with the points of $\hat{A}[2]$ on curves in $|\hat{L}|$. So, for notational simplicity we consider the fixed locus of $[-1_A]$ acting on moduli spaces of sheaves on $A$ instead.
We summarize the results of \cite{Frei_Honigs} on the isolated fixed points here.

The points of $K_A(0,l,s)$
that are fixed by $[-1_A]$ correspond to sheaves supported on curves in the eigenspaces of $|L|$ under the action of $[-1_A]$.
As described in \S\ref{base}, the eigenspaces of $H^0(A,L)$ are $1$ and $2$-dimensional, and these linear eigensystems consist of a single hyperelliptic genus $4$ curve $|L|^-=C$ and a pencil $|L|^+$ of non-hyperelliptic curves of arithmetic genus~$4$.
We may choose $L$ so that the intersection of $C$ with $A[2]$ is $10_1$ and the base locus of $|L|^+$ is $6_1$.
For each point $x\in 10_1$, there is a curve $B_x$ in $|L|^+$ that has a nodal singularity at $x$ \cite{Naruki}.

The supporting curves of sheaves in the fixed locus of $[-1_A]$ acting on 
$K_A(0,l,s)$ will be the same for any value of $s$. We 
reduce to analyzing $K_A(0,l,1)$ since
since the Abel--Jacobi map is surjective and generically one-to-one in degree $4=s+3$.
We may recover sheaves of $K_A(0,l,-1)$  from those of
$K_A(0,l,1)$ by subtracting $2\cdot 1$ from each divisor.

The set $R$ contains $16$ sheaves whose support is $C$.
$15$ of them correspond to divisors on $C$ consisting of
 four distinct points  of $10_1$ that sum to the identity, which we list in the second column of Table~\ref{correspond}. The last may be (nonuniquely) represented with the divisor~$4\cdot 1$. 

The remaining $20$ sheaves in $R$ are supported on the nodal curves $B_x\in |L|^+$. For each $x\in 10_1$, there are 
two ways to choose three points in $6_1$ that sum to $x$ in the group law of $A$. The two generalized divisors on $B_x$ defined by these choices of points
correspond to two sheaves in $R$. For instance in the case $x=1$, we have the following two generalized divisors:
\[
1+ab'+bc'+ca' \quad\text{and}\quad 1+ac'+ba'+cb'.
\]
The generalized divisors on $B_x$ are listed in the fourth column of Table~\ref{correspond}.

\begin{table}
  \centering
\caption{}
\label{correspond}
\begin{tabular}{|p{2.6cm}|p{2.5cm}||p{2.6cm}|p{2.8cm}|}
\hline
$\xi\in S$, $q(\xi)=0$, $D_{\Psi(\xi)}=C$ &
$r\in R$, $D_r=C$&
$\xi\in S$, $q(\xi)=1$, $D_{\Psi(\xi)}=B_x$&
$r\in R$, $D_r=B_x$
  \\\hline
               &  &$(ac',ba',cb')$&$1+ac'+ba'+cb'$\\ 
  &&$(ba',cc',ab')$ &$1+ba'+cc'+ab'$\\ 
$(a,a',aa')$&$1+a+a'+aa'$&&\\ 
$(b,b',bb')$&$1+b+b'+bb'$&$(a,b',ab')$ &$cc'+ca'+ba'+bc'$\\ 
$(c,c',cc')$&$1+c+c'+cc'$&$(a',b,ba')$ &$cc'+cb'+ab'+ac'$\\ 
$(a,b,c)$&$1+a+b+c$&$(a,c',ac')$ &$bb'+ba'+ca'+cb'$\\ 
$(a',b',c')$&$1+a'+b'+c'$&$(a',c,ca')$&$bb'+bc'+ab'+ac'$\\ 
$(aa',bb',cc')$&$1+aa'+bb'+cc'$&$(b,c',bc')$&$aa'+ab'+ca'+cb'$\\ 
                             &&$(b',c,cb')$ &$aa'+ac'+bc'+ba'$\\
$(c',ca',cb')$&$a'+b'+c+cc'$&$(a,bb',cb')$&$c'+ca'+ba'+ac'$\\              
$(a',ab',ac')$&$a+b'+c'+aa'$&$(b,aa',ca')$&$c'+cb'+bc'+ab'$\\              
$(b,cb',ab')$&$a+b'+c+bb'$&$(a',bb',bc')$&$c+ca'+ab'+ac'$\\  
$(b',bc',ba')$&$a'+b+c'+bb'$&$(b',aa',ac')$&$c+cb'+ba'+bc'$\\  
$(c,bc',ac')$&$a+b+c'+cc'$&$(a,cc',bc')$&$b'+ba'+ca'+ab'$\\
$(a,ba',ca')$&$a'+b+c+aa'$&$(c,aa',ba')$&$b'+ac'+bc'+cb'$\\  
$(aa',bc',cb')$&$a+a'+bb'+cc'$&$(a',cc',cb')$&$b+ba'+ac'+ab'$\\              
$(bb',ca',ac')$&$b+b'+aa'+cc'$&$(c',aa',ab')$&$b+bc'+ca'+cb'$\\           
  $(cc',ab',ba')$&$c+c'+aa'+bb'$&$(b,cc',ac')$&$a'+ab'+ba'+cb'$\\
&&$(c,bb',ab')$&$a'+ac'+bc'+ca'$\\
&&$(b',cc',ca')$  &$a+ab'+bc'+ba'$\\
&&$(c',bb',ba')$& $a+ac'+ca'+cb'$\\      
\hline  
\end{tabular}
\end{table}

\subsection{The bijection $\Psi$ and supporting curves}\label{results}

We now consider the bijection \eqref{ST} 
in the case where $d=3$.
Applying our earlier results, we deduce the supporting curve $D_{\Psi(\xi)}$ for each $\xi\in S$ and then examine this result in relation to the Hudson table (Figure~\ref{mult}).

\begin{prop}\label{support}
Suppose $\xi\in S'$. If $q(\xi)=0$, then 
  $D_{\Psi(\xi)}=C$, and if $q(\xi)=1$, then
$D_{\Psi(\xi)}=B_x$ for some $x\in 10_1$.

In the unique case where $\xi\in S\setminus S'$, $D_{\Psi(\xi)}=C$.
\end{prop}

\begin{proof}
As discussed in \S\ref{isokl}, for any $r\in R$, $D_r$ is either
$C$ or $B_x$ for $x\in 10_1$.
Since $C\cap A[2]=10_1$ and $B_x\cap A[2]=6_1\cup \{x\}$, the statement in the case where $\xi\in S'$ is an immediate consequence of Proposition~\ref{determinant}.
Since $R$ contains
$16$ sheaves that are supported on $C$ and the image of $S'$
under $\Psi$ 
has accounted for $15$ of them, the last case is forced.
\end{proof}  

We now remark on how the points in each $\xi\in S'$ are situated in the Hudson table of Figure~\ref{mult}
and then relate this information to 
the supporting curve of $\Psi(\xi)$.

There are fifteen $\xi\in S'$ where $q(\xi)=0$.
In six cases, all three points in $\xi$ are in $10_1$ and they 
occupy three distinct rows and columns of the Hudson table.
We illustrate this with $\xi=(c,c',cc')$ on the left of \eqref{C_examples}; we have placed stars at $c,c',cc'$ and dots elsewhere.
In the other nine cases, $\xi$ contains two points in $6_1$ -- one in the top row, one from the leftmost column of the Hudson table --
and one point in $10_1$, which is their product. 
For example, $\xi=(ba',bc',b')$ is on the right side of \eqref{C_examples}.
\begin{equation}\label{C_examples}
\begin{tabular}{c|ccc}
		$\cdot$&$\cdot$&$\cdot$&$\cdot$\\
	\hline
	$\cdot$&$\cdot$&$\ast$&$\cdot$\\
	$\cdot$&$\ast$&$\cdot$&$\cdot$\\
	$\cdot$&$\cdot$&$\cdot$&$\ast$
\end{tabular}  
\quad\text{or}\quad
\begin{tabular}{c|ccc}
			$\cdot$&$\cdot$&$\ast$&$\cdot$\\
			\hline
			$\cdot$&$\cdot$&$\cdot$&$\cdot$\\
			$\ast$&$\cdot$&$\ast$&$\cdot$\\
			$\cdot$&$\cdot$&$\cdot$&$\cdot$
		\end{tabular}
\end{equation}  
There are twenty $\xi\in S'$ where $q(\xi)=1$. In two cases, all three points in $\xi$ are in $6_1$:  
all three points are either in the leftmost column or in the top row; $\xi=(ab',bc',ca')$ is illustrated on the right in \eqref{Bx_examples}. In the other eighteen cases,  $\xi$ has two points in $10_1$ that share a row or column, and one point in $6_1$, which occupies a distinct column or row from the other points. We illustrate this pattern with $\xi=(ba',bb',c')$ on the left of \eqref{Bx_examples}. 
\begin{equation}\label{Bx_examples}
\begin{tabular}{c|ccc}
		$\cdot$&$\cdot$&$\cdot$&$\cdot$\\
	\hline
	$\cdot$&$\cdot$&$\cdot$&$\ast$\\
	$\ast$&$\cdot$&$\cdot$&$\cdot$\\
	$\cdot$&$\cdot$&$\cdot$&$\ast$
\end{tabular}  
  \quad\text{or}\quad
\begin{tabular}{c|ccc}
			$\cdot$&$\ast$&$\ast$&$\ast$\\
			\hline
			$\cdot$&$\cdot$&$\cdot$&$\cdot$\\
			$\cdot$&$\cdot$&$\cdot$&$\cdot$\\
			$\cdot$&$\cdot$&$\cdot$&$\cdot$
		\end{tabular}
\end{equation}

\begin{prop}\label{Bx}
Let $\xi=(u,v,w)\in S'$ such that $q(\xi)=1$. 
Let $x\in A[2]$ be the unique entry in the Hudson table
that is distinct from $u,v,w$ but shares a row or column with each of them.
Equivalently, $x$ is the unique element of $A[2]$ so that $\{xu,xv,xw\}\subseteq 6_1$.
Then, $D_{\Psi(\xi)}=B_x$.
\end{prop}

\begin{proof}
By Proposition~\ref{support}, since $q(\xi)=1$, we have
$D_{\Psi(\xi)}=B_x$ for some $x\in 10_1$.

Taking the approach of Sections \ref{curve} and \ref{notiso}, we know that $x$ is the unique element of $10_1$
so that \hyperlink{star}{$(\star)$} of \eqref{les} fails to be an isomorphism.

For any $x\in 10_1$,
the eigenvalues of $[-1_A]$ acting on $H^0(A,L\otimes P_x)$ are $1,1,-1$. 
The eigenvalues of $[-1_A]$ acting on
$H^0(A,\O_{\xi}\otimes L\otimes P_x)$ are determined by the quadratic form $q_x$. The multiplicities of $1$ and $-1$ are equal to $|\xi\cap 10_x|$ and
$|\xi\cap 6_x|$. In terms of the Hudson table, $|\xi\cap 6_x|$ is the number of points in $\xi$, other than $x$, that share a row or column with $x$.

From our above analysis of the Hudson tables, we know there is a unique $x\in 10_1$ so that $|\xi\cap 6_x|=3$, and thus the eigenvalues of $[-1_A]$ acting on $H^0(A,\O_{\xi}\otimes L\otimes P_x)$ are $-1$ with multiplicity $3$, so \hyperlink{star}{$(\star)$} cannot be an isomorphism and the result is proved.

The alternate characterization of $x$ in the statement of the proposition is an immediate consequence of Lemma~\ref{prod}.
\end{proof}

\begin{rmk}
The proofs of Propositions~\ref{support} and \ref{Bx} show that we are able to detect all of the $2$-torsion points in the support of the curves $\Psi(\xi), \xi\in S$ via differences in the eigenvalues of $[-1_A]$ acting on the domain and codomain of \hyperlink{star}{$(\star)$}.
In light of Remark~\ref{singular}, these proofs have detected all the singular points in these curves as well.
\end{rmk}  

We illustrate Proposition~\ref{Bx}
in each of the examples shown in \eqref{Bx_examples}
by placing a bullet at the values of $x$:
\[
\begin{tabular}{c|ccc}
		$\cdot$&$\cdot$&$\cdot$&$\cdot$\\
	\hline
	$\cdot$&$\cdot$&$\cdot$&$\ast$\\
	$\ast$&$\cdot$&$\cdot$&$\bullet$\\
	$\cdot$&$\cdot$&$\cdot$&$\ast$
\end{tabular}  
  \quad\text{or}\quad
\begin{tabular}{c|ccc}
			$\bullet$&$\ast$&$\ast$&$\ast$\\
			\hline
			$\cdot$&$\cdot$&$\cdot$&$\cdot$\\
			$\cdot$&$\cdot$&$\cdot$&$\cdot$\\
			$\cdot$&$\cdot$&$\cdot$&$\cdot$
		\end{tabular}
              \]
In Table~\ref{correspond}, 
we show the $3$ points on $A$ that are in $S'$ side by side with
divisors correpsonding to 
sheaves in $R$ that are supported on the same curve as their images under $\Psi$.
In the two columns on the right, the points are grouped by the values of $x$ in $B_x$, which may be read off from the first summand of each divisor in the rightmost column.

\begin{rmk}\label{bijection} It is also possible to formulate the specific correspondence between $S$ and $R$ listed in Table~\ref{correspond} in terms of the Hudson table, which we conjecture is the bijection given by $\Psi$.

Let $\xi=(u,v,w)\in S'$.
There is a point $z\in A[2]$ so that $(u,v,w)$ is paired with the $z+zu+zv+zw$ in Table~\ref{correspond}, which is selected in the following way:
If $q(\xi)=1$ then $z\in A[2]$ is the unique point so that $\{zu,zv,zw\}\subseteq 6_1$ (Proposition~\ref{Bx}).
If $q(\xi)=0$ then $z\in A[2]$ is a point such that $\{zu,zv,zw\}\subseteq 10_1$. Although $z$ is not always unique, the divisor $z+zu+zv+zw$ will be the same for any such choice of $z$.
When $\{u,v,w\}\subseteq 10_1$ are in three distinct rows and three distinct columns in the Hudson table, $z$ is unique, and, in fact, must be $1$.
 In the case where $u\in 10_1$ and $v,w,\in 6_1$, there are four possible values of $z$. The possible values of $z$ in the examples \eqref{C_examples} are marked with diamonds below:
\[
\begin{tabular}{c|ccc}
		$\diamond$&$\cdot$&$\cdot$&$\cdot$\\
	\hline
	$\cdot$&$\cdot$&$\ast$&$\cdot$\\
	$\cdot$&$\ast$&$\cdot$&$\cdot$\\
	$\cdot$&$\cdot$&$\cdot$&$\ast$
\end{tabular}  
\quad\text{or}\quad
\begin{tabular}{c|ccc}
			$\cdot$&$\cdot$&$\ast$&$\cdot$\\
			\hline
			$\cdot$&$\diamond$&$\cdot$&$\diamond$\\
			$\ast$&$\cdot$&$\ast$&$\cdot$\\
			$\cdot$&$\diamond$&$\cdot$&$\diamond$
\end{tabular}
\]
\end{rmk}

\section{Action of the symplectic group}\label{sp}

Let $A$ be an abelian surface with a symmetric polarization $L$ of odd degree.
As discussed in \S\ref{label}, the Weil pairing gives a nondegenerate, strictly alternating form on $A[2]$. In this section we examine the action of its symplectic group on the set $S'$. Some references for material in this section are \cite{GrossHarris,Farkas}.

\subsection{Action of $\Spa$ on $S'$}
Let $\Spa$ be the group of all $\F_2$-linear transformations $T:A[2]\to A[2]$ that preserve the Weil pairing. $\Spa$ has order $720$ and is generated bythe $16$ transvections from each $u\in A[2]$:
\[
T_u(v)=v+\langle v,u\rangle u.
\]  
The group $\Spa$ also acts on theta characteristics of $(A,L)$.
If $q$ is a theta characteristic and $T\in \Spa$, then  $Tq(Tv)=q(v)$ for each $v\in A[2]$.

There is a natural action of $\Spa$ on $S'$: if $\xi=(u_1,\ldots,u_d)$, then $T(\xi)=(Tu_1,\ldots,Tu_d)$.
The points $Tu_1,\ldots,Tu_d$ are distinct since $T$ is an isomorphism and 
since addition on $A[2]$ as an $\F_2$ vector space coincides with the group operation of $A$, the elements
$Tu_1,\ldots,Tu_d$ will also sum to the identity.

Many of our results depend on the interaction between  $S'$ and theta characteristics. The action of $\Spa$ preserves these interactions, as we now show.

\begin{lemma}
Let $\xi\in S'$ and $T\in\Spa$. Then $q(\xi)=q(T\xi)$.
\end{lemma}

\begin{proof}
Let $\xi=(u_1,\ldots,u_d)$.
Pick a theta characteristic $q_v$.
By definition,
$q(\xi)=q_v(u_1)+\cdots+q_v(u_d)$.
Then, since we may use any choice of
theta characteristic to determine
$q(T\xi)$, we have
$$q_v(u_1)+\cdots+q_v(u_d)=
Tq_v(Tu_1)+\cdots+Tq_v(Tu_d)=q(T\xi),$$
concluding the proof.
\end{proof}

\begin{prop}\label{Tq} Let $\xi\in S'$, $T\in \Spa$
  and $q_v$ a theta characteristic.
  Let
$L\otimes P_x$ and $L\otimes P_y$ be
  the line bundles corresponding to $q_v$ and $Tq_v$. 
We construct long exact sequences  as in \eqref{les}
using $\xi,L\otimes P_x$ and $T(\xi),L\otimes P_y$, which contain the following maps in the position \hyperlink{star}{$(\star)$}:
\begin{align*}
\varphi:H^0(A,L\otimes P_x)
&\to 
  H^0(A,\O_{\xi}\otimes L\otimes P_x),
\\
\psi:H^0(A,L\otimes P_y)
&\to 
  H^0(A,\O_{T(\xi)}\otimes L\otimes P_y).
\end{align*}
The eigenvalues of $[-1_A]$ acting on the domains of 
$\varphi$ and $\psi$ are the same.
The eigenvalues of the codomains are also the same.
\end{prop}

In particular, if $\varphi$ cannot be an isomorphism due to the eigenvalues of $[-1_A]$ occurring with different multiplicities in its domain and codomain, then $\psi$ will also not be an isomorphism.

\begin{proof}
  As described in \S\ref{base}, the eigenvalues of the domains of $\varphi$ and $\psi$ are determined by whether the theta characteristics $q_v$ and $Tq_v$ are even or odd. The action of $\Spa$ on theta characteristics preserves the Arf invariant
  \cite[Proposition~1.11]{GrossHarris}, and thus we have the result for the domains.

  The multiplicities of the eigenvalues $1$ and $-1$ on the codomain of $\varphi$ are $\#\{i\mid q_v(u_i)=0\}$ and $\#\{i\mid q_v(u_i)=1\}$, respectively. The multiplicities
  on the codomain of $\psi$ are 
$\#\{i\mid Tq_v(Tu_i)=0\}$ and $\#\{i\mid Tq_v(Tu_i)=1\}$.
Since $Tq_v(Tu_i)=q_v(u_i)$, we have the result. 
\end{proof}

\subsection{Action of $\Spa$ in the $d=3$ case}

We now assume that $L$ is a $(1,3)$-polarization and examine the orbits of
$\Spa$ acting on $S'$.

\begin{prop}\label{twoorbits}
The action of $\Spa$ on $S'$ has two orbits, which consist of $\xi\in S'$ where $q(\xi)=0$ and $q(\xi)=1$, respectively.
\end{prop}

\begin{proof}
  We first examine the elements  $\xi\in S'$ where $q(\xi)=1$.
  As described in Proposition~\ref{Bx}, for each such $\xi$ there is a unique $x\in 10_1$ so that the map \hyperlink{star}{$(\star)$} associated to $\xi$ and $q_x$ is not an isomorphism, which is witnessed by the eigenvalues of $[-1_A]$ acting on the domain and codomain of \hyperlink{star}{$(\star)$}.
For example, for $(ab',bc',ca')\in S'$, which
illustrated on the right-hand side of \eqref{Bx_examples},
 we have $x=1$.

 For any $v\in 10_1$, $T_v(q_1)=q_v$. We may see this from observing that $T_v(10_1)=10_v$ and $T_vq_1(T_v(10_1))=q_1(10_1)$. Thus for each $v\in 10_1$,
 $\Psi(T_v(ab',bc',ca'))$ is supported on the curve $B_v$.

 The composition of transvections $T_c\circ T_b\circ T_a$ fixes $q_v$ for $v\in 10_1$ but exchanges the two sets of points $\{ab',bc',ca'\}$ and $\{ac',ba',cb'\}$. Thus we see that the $20$ points
 $\xi\in S'$ where $q(\xi)=1$ form one orbit under the action of $\Spa$.

 It can be checked directly by successively applying transvections that the remaining points also form an orbit.
\end{proof}

\section{Some examples in the $(1,5)$-polarized case}\label{onefive}
It is possible to use the methods of sections \ref{curve}, \ref{notiso} and \ref{sp}
to analyze the bijection $\Psi:S\to R$ for any odd $d$ and make deductions about singularities at points in $\hat{A}[2]$ occurring in the supporting curves of $R$.
In this section, we show
some consequences of these ideas for 
examples of $\xi\in S$ when $d=5$.

Let $A$ be an
abelian surface that is $(1,5)$-polarized by a symmetric line bundle $L$ so that $\NS(A)=\Z l$, where $l=[L]$. We assume that $L$ is associated to an even theta characteristic $q_1$.
In this setting, $K_4A$ and $K_{\hat{A}}(0,l,-1)$ are $8$-folds, and the linear system $|L|$ contains curves of genus $6$. The vector space $H^0(A,L)$ is $5$-dimensional. Its two eigenspaces under the action of $[-1_A]$ are
a net $|L|^+$ and
a pencil $|L|^-$ with base loci $6_1$ and $10_1$, resp.

\begin{example}\label{ex.on}
  Consider the point $\xi=(ab',a',c,b,c')\in S$. Since $q(\xi)=1$, by Proposition~\ref{determinant}, $D_{\Psi(\xi)}\supseteq 10_1$.

  Moreover, for any $x\in 10_1$, the eigenvalues of $[-1_A]$ acting on $H^0(A,L\otimes P_x)$ are $1$ with multiplicity $3$ and
  $-1$ with multiplicity $2$. The eigenvalues of
  $[-1_A]$ acting on $H^0(A,\O_{\xi}\otimes L\otimes P_x)$ are either
$1$ with mult.\ $1$ and $-1$ with mult.\ $4$
or
$1$ with mult.\ $3$ and $-1$ with mult.\ $2$, so it is possible the supporting curve $D_{\Psi(\xi)}$ is smooth at all the points of $10_1$.

Examining \hyperlink{star}{$(\star)$} for $x\in 6_1$, the eigenvalues of $[-1_A]$ occur with identical multiplicities in the domain and codomain, except when $x=ba'$.
Thus, $D_{\Psi(\xi)}\in |L|^-$ and has a singular point $ba'$.
\end{example}

\begin{example}\label{ex.tw}
  Consider the point $\xi=(ab',cb',a',bb',c')\in S$.
  Since $q(\xi)=0$, we have $6_1\subseteq D_{\Psi(\xi)}$.

  When $x\in 6_1$,
  the eigenvalues of $[-1_A]$ acting on $H^0(\O_{\xi}\otimes L\otimes P_x)$
  are $1$ with mult.\ $1$ and $-1$ with mult.\ $4$
  or $1$ with mult.\ $3$ and $-1$ with mult.\ $2$, which does not exclude the possibility that the points $6_1$ are smooth.

When $x\in 10_1$, the eigenvalues of $[-1_A]$ acting on the domain and codomain of
\hyperlink{star}{$(\star)$} differ when $x=b$ and $x=b'$.

Thus $D_{\Psi(\xi)}$ is contained in the net $|L|^+$. It also passes through the points $b$ and $b'$, and it is singular at both of those points.
\end{example}

\begin{example}\label{ex.th}
  Finally, we consider $\xi=(1,ac',ab',b',c')$. Again, $q(\xi)=0$, so $6_1\subseteq \Psi(\xi)$ and $D_{\Psi(\xi)}$ is contained in the net $|L|^+$.

However, when $x=b'$ or $x=c'$,  
the eigenvalues of $[-1_A]$ acting on $H^0(\O_{\xi}\otimes L\otimes P_x)$
are $1$ with mult.\ $5$, and so
$b'$ and $c'$ are singular points of~$D_{\Psi(\xi)}$.
\end{example}

We show Hudson tables below illustrating Examples~\ref{ex.on}, \ref{ex.tw}, and \ref{ex.th}. The points in $\xi$ are marked with stars and
the singular points on $D_{\Psi(\xi)}$ marked with bullets. In the rightmost table, two of the entries are marked with both bullets and stars.
\[
\begin{tabular}{c|ccc}
		$\cdot$&$\ast$&$\cdot$&$\cdot$\\
	\hline
	$\cdot$&$\ast$&$\ast$&$\cdot$\\
	$\bullet$&$\cdot$&$\cdot$&$\cdot$\\
	$\cdot$&$\ast$&$\cdot$&$\ast$
\end{tabular}  
\quad\quad
\begin{tabular}{c|ccc}
			$\cdot$&$\ast$&$\cdot$&$\cdot$\\
			\hline
			$\cdot$&$\ast$&$\cdot$&$\ast$\\
			$\cdot$&$\cdot$&$\bullet$&$\cdot$\\
			$\ast$&$\bullet$&$\cdot$&$\ast$
\end{tabular}
\quad\quad
\begin{tabular}{c|ccc}
			$\ast$&$\ast$&$\cdot$&$\cdot$\\
			\hline
			$\ast$&$\cdot$&$\cdot$&$\cdot$\\
  $\cdot$&$\cdot$&
$\ast$\hspace*{-1pt}\llap{\raisebox{1.2pt}{$\scriptscriptstyle\bullet$}}
                   &$\cdot$\\
  $\cdot$&$\cdot$&$\cdot$
                                             &
$\ast$\hspace*{-1pt}\llap{\raisebox{1.2pt}{$\scriptscriptstyle\bullet$}}
\end{tabular}
\]  
Although the $\xi$ in Examples~\ref{ex.tw} and \ref{ex.th}
have the same value of $q(\xi)$ and 
the curves $D_{\Psi(\xi)}$  both contain two singular points, they cannot be in the same orbit of $S'$ under the action of $\Spa$. 
By Proposition~\ref{Tq} and its proof, if we write $\xi=(u_1,\ldots,u_d)$, then the following sets of multiplicities are preserved by the action of $\Spa$:
$$\{\#\{i\mid q_v(u_i)=0\}\mid v\in A[2]\}\quad\text{and}\quad\{\#\{i\mid q_v(u_i)=1\mid v\in A[2]\}\}.$$
We see from the Hudson tables above that these values must differ between Examples~\ref{ex.tw} and \ref{ex.th} since the numbers of entries in the Hudson table that do not share a row or column with any $u_i\in \xi$ are different overall.

So, unlike the $d=3$ case, in the $d=5$ case, $S'$ has more than two orbits
under the action of $\Spa$. 

\bibliographystyle{alpha}
\bibliography{mainbib}

\end{document}